\documentclass[reqno]{gokovaalt}
\usepackage{epsfig,graphicx}

\usepackage{textcmds} 
\usepackage{amsmath, amssymb, amsfonts, amstext, amsthm, amscd, mathrsfs, mathscinet}
\usepackage{mathtools}
\usepackage{verbatim}
\usepackage{graphicx}
\usepackage{color}
\usepackage{pinlabel}
\usepackage{mathrsfs} 
\usepackage{enumitem}
\usepackage{tikz}
\usetikzlibrary{matrix}

\usepackage{hyperref}

\newcommand{\C}{\mathbb{C}}
\newcommand{\Z}{\mathbb{Z}}
\newcommand{\R}{\mathbb{R}}

\newcommand{\RP}{\mathbb{R}P}

\newcommand{\CP}{\mathbb{C}P}
\newcommand{\T}{\mathbb{T}}

\newcommand{\id}{\mathrm{Id}}

\newcommand{\del}{\partial}
\newcommand{\delbar}{\overline{\partial}}

\newcommand{\OP}{\operatorname}

\renewcommand{\Im}[1]{\mathfrak{Im}\,#1}

\renewcommand{\Im}{\mathfrak{Im}}

\newsavebox{\textvisiblespacebox}
\savebox{\textvisiblespacebox}{\texttt{aa}}
\newcommand\vartextvisiblespace[1][\wd\textvisiblespacebox]{%
 \makebox[#1]{\kern.1em\rule{.4pt}{.3ex}%
 \hrulefill%
 \rule{.4pt}{.3ex}\kern.1em}%
}

\numberwithin{equation}{section}

\newtheorem{thm}{Theorem}[section]
\newtheorem{lma}[thm]{Lemma}
\newtheorem{prp}[thm]{Proposition}

\newtheoremstyle{TheoremNum}
  {\topsep}{\topsep}       
  {\itshape}           
  {}               
  {\bfseries}           
  {.}               
  { }               
  {\thmname{#1}\thmnote{ \bfseries #3}}
\theoremstyle{TheoremNum}

\theoremstyle{definition}
\newtheorem{dfn}[thm]{Definition}

\newtheorem{quest}[thm]{Question}

\newtheorem{ex}[thm]{Example}

\theoremstyle{remark}
\newtheorem{rmk}[thm]{Remark}

\begingroup 
\makeatletter 
\@for\theoremstyle:=definition,remark,plain,TheoremNum\do{%
\expandafter\g@addto@macro\csname th@\theoremstyle\endcsname{%
\addtolength\thm@preskip\parskip 
}%
} 
\endgroup 

\title{Construction of symplectic flexible links}

\author{Johan Björklund}
\address{
Avdelningen för elektroteknik, matematik och naturvetenskap\\
Högskolan i Gävle\\
SE-801 76 GÄVLE\\
SWEDEN}
\email{johan.bjorklund@hig.se}

\author{Georgios Dimitroglou Rizell}
\address{Department of Mathematics\\
Uppsala University\\
Box 480\\
SE-751 06 UPPSALA\\
SWEDEN}
\email{georgios.dimitroglou@math.uu.se}
\thanks{The author is supported by the Knut and Alice Wallenberg Foundation through the grants KAW 2021.0191 and KAW 2023.0294.}

\begin{document}

\begin{abstract}
We show that any smooth one-dimensional link in the real projective three-plane is the fixed-point locus of a smooth symplectic surface in the complex projective three-plane which is invariant under complex conjugation. The degree of the surface can be taken to be either one or two, depending on the homology class of the link. In other words, there are no obstructions to finding a symplectic representative of a flexible link beyond the classical topology.
\end{abstract}
\keywords{Real algebraic knots and links, flexible knots and links, anti-symplectic involution, symplectic submanifolds, transverse knots}

\maketitle

\section{Introduction and results}

A smooth link $L \subset \RP^3 \subset \CP^3$ is called \textbf{real algebraic} if it is the real locus of a closed real smooth algebraic curve $\Sigma \subset \CP^3$, i.e.~an embedded algebraic surface that is invariant under complex conjugation $I \colon \CP^3 \to \CP^3$, i.e.~ the real part of an embedded algebraic curve that is invariant under complex conjugation. For any submanifold $\Sigma \subset \CP^n$ which is invariant under complex conjugation $I$, we write $\Sigma_\R=\Sigma \cap \RP^3=L$ for its fixed-point locus. Recall that $I$ is an anti-symplectic involution for the standard Fubini--Study symplectic form $\omega_{FS}$, that is $I^*\omega_{FS}=-\omega_{FS}$, and that the fixed-point locus of $I$ is precisely $\RP^3$. An $I$-invariant surface $\Sigma \subset \CP^3$ is said to be of \textbf{Type-I} if $\Sigma \setminus \Sigma_\R$ is disconnected, and is otherwise said to be of \textbf{Type-II}.

 The first author classified rational real algebraic knots up to degree five in \cite{Johan:Real}. The cases of degree less than six, and genus at most one, were later classified by Mikhalkin--Orevkov in \cite{MikhalkinOrevkov}. Also, see work by Akbulut \cite{Akbulut} for the existence of complexifications of knots (that typically live in high degrees). Real algebraic knots are naturally very rigid objects. For instance, it is easy to show that the isotopy class of $\Sigma_\R$ is unique when $\Sigma$ is of degree one or two, even up to isotopy through algebraic curves. On the contrary, dropping the condition on $\Sigma$ being algebraic and only requiring the surface to be orientable, $I$-invariant, and compatible with the complex orientation, one obtains a much more flexible type of object that was studied by the first author in \cite{Johan:Flexible}. Such a link in $\RP^3$ is called a \textbf{flexible link} and, as the name suggests, flexible links are determined by classical topological invariants and the knot type up to flexible isotopy; see the aforementioned work.

Compared to the real algebraic case, there is a big difference in the flexible case; namely, in this case the real locus does not determine the $I$-invariant surface, not even its topology. For this reason we also include the choice of the $I$-invariant surface itself as a part of the definition of a flexible link. In this note we are concerned with the following more restrictive subclass:
\begin{dfn}
A \textbf{flexible-symplectic link} is an embedded symplectic $I$-invariant surface $\Sigma \subset (\CP^3,\omega_{FS})$ of real dimension $\dim\Sigma=2$. We denote its fixed-point set by $\R \Sigma=\Sigma \cap \RP^3$, which is a smooth link.
\end{dfn}
The reason why we again call this object ``flexible'' is that the existence of an extension of a link $L \subset \RP^3$ to an $I$-invariant symplectic surface only is governed by classical topology. Establishing this fact is the main goal of our paper. Before stating the result, we point out the obvious fact that any closed symplectic surface $\Sigma \subset \CP^3$ must have a positive degree $d > 0$, i.e.~must be homologous to $[\Sigma]=d[\CP^1] \in \Z[\CP^1]=H_2(\CP^3)$ with $d > 0$.
\begin{thm}
\label{thm:main}
 Any smooth link $L \subset \RP^3$ can be realised as the real-part $L=\Sigma_\R$ of a connected flexible-symplectic link of degree $d \in \{1,2\}$ of the same parity as $[L] \in H_1(\RP^3)=\Z_2$. Furthermore, the surface $\Sigma \setminus L$ can be either of Type I or II; in the first case the genus can assume any value $g= \OP{rk} H_0(L)-1+2k$ with $k\ge 0$, while in the second case it can assume any value $g= \OP{rk} H_0(L)-1+k$ with $k \ge 1$. 
\end{thm}
We plan to study the classification question in future work, i.e.~the set of flexible-symplectic links up to the relation of isotopy through links of the same type.
\begin{rmk}
Any $I$-invariant surface $\Sigma\subset \CP^3$ of real dimension $\dim\Sigma=2$ that intersects $\RP^3$ cleanly must have a degree of the same parity as that of the corresponding class in $[\Sigma \cap \RP^3] \in H_1(\RP^3)=\Z_2$. Namely, the divisor $\CP^2_\infty \subset \CP^3$ is fixed set-wise by $I$ with fixed-point locus $\RP^2_\infty \subset \RP^3$, and the degrees of a curve in $\CP^3$ and $\RP^3$ are given by the intersection number with $\CP^2_\infty$ and $\RP^2_\infty$, respectively.
\end{rmk}

\section{Background on symplectic and contact geometry}

\subsection{Symplectic manifolds}

A \textbf{symplectic manifold} is a pair $(X^{2n},\omega)$ consisting of a smooth even-dimensional manifold $X^{2n}$ together with a closed non-degenerate two-form $\omega \in \Omega^2(X)$, i.e.~a \textbf{symplectic two-form}.
\begin{ex}Here we are mainly interested in the complex projective space $(\CP^n,\omega_{FS})$ equipped with the Fubini--Study K\"{a}hler form $\omega_{FS}\in \Omega^{(1,1)}(\CP^n)$. This symplectic manifold has an integrable complex structure $J \in End(T\CP^n)$, $J^2=-\id$, and the the symplectic form is a K\"{a}hler form which in local affine coordinates can be written as
$$\omega_{FS}=\frac{i}{2}\del\delbar\log{\left(1+|\mathbf{z}|^2\right)}=-\frac{1}{4}dd^c\log{\left(1+|\mathbf{z}|^2\right)}, \:\:d^c f \coloneqq d f\circ J,$$
i.e.~it has a locally defined K\"{a}hler potential $\log{\left(1+|\mathbf{z}|^2\right)}$ in each standard affine coordinate patch. Note that the area of a line $\CP^1 \subset \CP^n$ is equal to $\int_{\CP^1} \omega_{FS}=\pi$ with our conventions.
\end{ex}
A \textbf{symplectomorphism} is a diffeomorphism between symplectic manifolds that preserves the symplectic forms. A smooth isotopy $\phi^t$ through symplectomorphisms is a \textbf{Hamiltonian isotopy} if $\iota_{\dot{\phi}^t}\omega=-dH_t$ for some smooth function $H_t \colon X \to \R$ which is uniquely defined up to the addition of a locally constant function. The time-one map of a Hamiltonian isotopy is called a \textbf{Hamiltonian diffeomorphism}.
\begin{ex}The group of K\"{a}hler isometries of $(\CP^n,\omega_{FS})$ is equal to the projective unitary group $PU(n+1)$, which acts linearly on the homogeneous coordinates $\C^{n+1}\setminus \{0\}$ by elements in $U(n+1)$. Since these groups are connected and $\CP^n$ is simply connected, the elements consist of Hamiltonian diffeomorphisms.
\end{ex}
 
A submanifold $\iota \colon Y \hookrightarrow (X,\omega)$ is said to be \textbf{symplectic} if the two-form $\iota^*\omega$ is symplectic. A standard fact about isotopies of symplectic submanifolds that will be crucial for us is the ``Hamiltonian Thom Isotopy Lemma'', which is a folklore result that was given a thorough proof by Antonini--Cavalletti--Lerario in \cite{HamiltonianThom}:
\begin{thm}[The Hamiltonian Thom Isotopy Lemma; e.g.~\cite{HamiltonianThom}]
\label{thm:HamiltonianThom}
 A smooth isotopy $\iota_t \colon (Y,\iota_t^*\omega) \hookrightarrow (X,\omega)$ of symplectic embeddings which is fixed outside of a compact subset of the domain is generated by a compactly supported Hamiltonian isotopy, in the sense that $\iota_t(Y)=\phi^t\circ \iota_0(Y)$, where $\phi^t$ is a compactly supported Hamiltonian isotopy of $(X,\omega)$.
\end{thm}
\begin{rmk}
 
 \begin{enumerate}
  \item The proof in \cite{HamiltonianThom} is for isotopies of compact symplectic submanifolds of compact symplectic manifolds. However, the proof immediately works also for compactly supported isotopies of properly embedded symplectic manifolds. The latter version of the result translates to the version that we stated above.
  \item Since we only make requirements about the images $\phi^t\circ \iota_0(Y)=\iota_t(Y)$, and not the parametrised maps, there is no need to assume that $Y$ is simply connected. Note that the two parametrised symplectic embeddings $\phi^t\circ \iota_0$ and $\iota_t$ of the domain $(Y,\omega_Y)$ thus very well may differ by a non-Hamiltonian symplectic isotopy of $Y$.
 \end{enumerate}
\end{rmk}

 Recall that a submanifold $L \subset (X^{2n},\omega)$ is said to be \textbf{isotropic} if $\omega$ pulls back to zero on $L$. This implies that $L$ is at most half-dimensional. If the dimension, moreover, is equal to half i.e.~$\dim L=n$, then we say that $L$ is \textbf{Lagrangian}.

\begin{ex}
The one-form $d^c\log{(1+|\mathbf{z}|^2)}$ vanishes along the real part $\RP^n \subset \CP^n$ in each standard affine coordinate patch. In other words, $\RP^n \subset (\CP^n,\omega_{FS})$ is a Lagrangian submanifold. 
\end{ex}

 When $I$ is an \textbf{anti-symplectic involution} of $(X,\omega)$, i.e.~$I^*\omega=-\omega$, it follows that the fixed-point set of $I$ is a Lagrangian submanifold.

\begin{ex}
The pull-back of the Fubini--Study form under complex conjugation is equal to
$$ I^* \frac{i}{2}\del\delbar\log{\left(1+|\mathbf{z}|^2\right)}=\frac{i}{2}\delbar\del\log{\left(1+|\mathbf{z}|^2\right)}=-\omega_{FS},$$
since the operators $\del$ and $\delbar$ anti-commute. In other words, $I$ is an anti-symplectic involution. Since $\RP^n$ is the set of fixed-points of this involution, this also implies that $\RP^n$ is Lagrangian.
\end{ex}

\subsection{Contact manifolds}

Contact geometry can be expressed in terms of conical symplectic geometry. More precisely, a \textbf{contact manifold} is a pair $(Y^{2n+1},\xi)$ that consists of an odd-dimensional manifold together with a maximally non-degenerate field of tangent hyperplanes $\xi \subset TY$. We will be interested in the case when $\xi=\ker \alpha$ for some auxiliary contact one-form $\alpha \in \Omega^1(Y)$. The contact condition can be expressed as the condition that the conical manifold
$$ (\R_\tau \times Y,d(e^\tau\alpha))$$
is symplectic.

A one-dimensional knot or link in a contact manifold $T \subset (Y,\xi)$ is said to be \textbf{transverse} if it is everywhere transverse to the contact hyperplane distribution $\xi \subset TM$. This condition can be equivalently formulated as the condition that
$$ \R \times T \subset (\R_\tau \times Y,d(e^\tau\alpha)) $$
is a symplectic cone. An oriented knot is said to be \textbf{positively transverse} if $\alpha$ is positive on its oriented tangent line. The high-dimensional analogue of a transverse submanifold is a submanifold $T \subset (Y,\alpha)$ onto which $\alpha$ restricts to a contact form. In particular, the $h$-principle for iso-contact embeddings of high codimension implies that
\begin{prp}[\cite{Eliashberg:IntroductionH}]
\label{prp:hprinciple}
Two positively transverse one-dimensional links in a $(2n-1)$-dimensional contact manifold with $n \ge 3$, say $L_i \subset (Y^{2n-1},\alpha)$ for $i=0,1$, are isotopic through transverse links if and only if they are homotopic.

Furthermore, when $L_i \coloneqq T_i \sqcup I_Y(T_i)$, $i=0,1,$ for some fixed-point free contact involution
$$I_Y \colon (Y^{2n-1},\alpha) \to (Y^{2n-1},-\alpha),$$
and if $T_0$ and $T_1$ are homotopic, then there exists a transverse isotopy from $L_0$ to $L_1$ through $I_Y$-invariant links.
\end{prp}
\begin{proof}
The first part is a direct application of the $h$-principle for iso-contact immersions \cite[24.1.2]{Eliashberg:IntroductionH}.

To construct the $I_Y$-equivariant transverse isotopy we argue as follows. First use the aforementioned $h$-principle to create a transverse isotopy $T_t$ from $T_0$ to $T_1$. It is then enough to perturb this isotopy to one which is in general position, to ensure that $T_t \sqcup I_Y(T_t)$ remains embedded for all $t$.
\end{proof}

An $(n-1)$-dimensional submanifold $\Lambda \subset (Y,\xi)$ of a $(2n-1)$-dimensional contact manifold is \textbf{Legendrian} if it is tangent to the contact distribution $\xi$ (this is the maximal possible dimension since $\xi$ is maximally non-integrable by definition of the contact condition). Equivalently, this can be formulated as the condition that
$$ \R \times \Lambda \subset (\R_\tau \times Y,d(e^\tau\alpha))$$
is a Lagrangian cone.

\subsection{Cotangent bundles}

Any cotangent bundle $T^*M$ has the \textbf{tautological symplectic structure} defined by the exact two-form $d\theta_M$. Here $\theta_M \in \Omega^1(T^*M)$ is the \textbf{tautological one-form}, which has the following local expression. Given local coordinates $(q^i,p_i)$ on $T^*M$, where $q^i$ are any local coordinates on $M$ and the $p_i$ are the coordinates induced by $q^i$ as the coefficients of the locally defined one-forms $p_i\,dq^i \in T^*M$, in other words $p_i$ are locally defined conjugate momentum coordinates, we can write $\theta_M=\sum_i p_i\,dq_i$. The cotangent fibres $T^*_pM$ as well as the zero-section $0_M \subset T^*M$ are Lagrangian with respect to the tautological symplectic structure. Furthermore, any diffeomorphism $\phi$ of $M$ lifts to a symplectic bundle-morphism of $(T^*M,d\theta_M)$ through the pull-back $(\phi^{-1})^*$ of the inverse. In fact, $(\phi^{-1})^*\theta_M=\theta_M,$ i.e.~the diffeomorphisms act on $T^*M$ while preserving the tautological one-form.

Any cotangent bundle $(T^*M,d\theta_M)$ has a canonical anti-symplectic involution $I_0$ given by fibre-wise multiplication by $-1$.

 Given a compact subset $U \subset T^*M$ that is fibre-wise star-shaped with respect to the origin in the fibres, e.g.~the radius-$r$ co-disc bundle $D^*_{g,r}M$ for some Riemannian metric $g$ on $M$, its boundary $Y=\partial U \subset T^*M$ is a contact manifold with contact form $\alpha \coloneqq \theta_M|_{TY}.$ Furthermore, the symplectisation of $(Y,\alpha)$ has a natural embedding in $T^*M$ via
\begin{gather*} (\R_\tau \times Y,e^\tau\alpha) \xrightarrow{\cong} (T^*M \setminus 0_M,\theta_M)\\
(\tau,y) \mapsto e^\tau\cdot y,
\end{gather*}
which preserves the primitives of the symplectic forms.

In the case when the star-shape subset $U$ moreover is invariant under the anti-symplectic involution $I$, it follows that $I_0$ is induced by a co-orientation reversing contactomorphism without fixed points 
$$I_Y \colon (Y,\ker \alpha) 
\to (Y,\ker \alpha)$$
that satisfies $I_Y^*\alpha=-\alpha$. Namely, $I$ is given by extending $I_Y$ to $T^*M \setminus 0_M \cong \R_\tau \times Y$ by $\id_\R$ on the $\R_\tau$-factor.

\subsection{Isotropic, Lagrangian, and symplectic cobordisms}

\label{sec:cob}

It will be necessary to consider also cobordisms embedded in
$$(\R_\tau \times Y,d(e^t\alpha)) $$
beyond simply trivial cones over submanifolds in $Y$. We say that a properly embedded submanifold $C$ of the symplectisation has \textbf{conical ends} if it coincides with
$$ (-\infty,-1] \times K_- \:\: \sqcup \:\: [1,+\infty) \times K_+ \subset (\R_\tau \times Y,d(e^t\alpha)) $$
outside of a compact subset. In this case we say that $C$ is a \textbf{cobordism from $K_- \subset Y$ to $K_+ \subset Y$ with conical ends}.

A cobordism $C_0$ from $K_-$ to $K$, and a cobordism $C_1$ from $K$ to $K_+$ can be concatenated in a manner to form a cobordism from $K_-$ to $K_+$. This is done by first translating $C_0$ (resp. $C_1$) sufficiently far in the negative (resp. positive) $\R_\tau$-direction, so that its positive (resp. negative) end contains the level $\{\tau=0\}$. After this has been done, we can form the concatenation in the following way:
\begin{dfn}
\label{dfn:concatenation}
After a translation as above, the union
$$C_0 \odot C_1 \coloneqq C_0\cap \{\tau \le 0\} \: \cup \: C_1 \cap \{\tau \ge 0\} $$
is a cobordism from $K_-$ to $K_+$ which is called the \textbf{concatenation} of $C_0$ and $C_1$.
\end{dfn}

We recall the classical Weinstein's Lagrangian neighbourhood theorem as well as a mild generalisation to the equivariant setting in the presence of an anti-symplectic involution.
\begin{thm}[Weinstein's Lagrangian Neighbourhood Theorem \cite{Weinstein:Lagrangian}]
\label{thm:weinstein}
Any proper Lagrangian embedding $\phi \colon L \hookrightarrow (X,\omega)$ can be extended to a symplectic embedding
$$\Phi \colon (D^*L,d(p\,dq)) \hookrightarrow (X,\omega), \:\: \Phi|_{0_L}=\phi,$$
where $D^*L$ is the unit cotangent bundle for some choice of Riemannian metric on $L$.

In addition,
\begin{enumerate}
 \item Let $(X,\omega)=(\R_\tau \times Y,d(e^\tau\alpha))$, and assume that there exists an embedding $\phi_\Lambda \colon \Lambda \hookrightarrow Y$ and coordinates $(s,x) \in I \times \Lambda \hookrightarrow L$, with $I \subset \R_s$ connected, such that $(\tau,y)=\phi(s,x)=(s,\phi_\Lambda(x))$. Using $\phi^t_{\partial_s}$ and $\phi^t_{\partial_\tau}$ to denote the time-$t$ flow of the respective coordinate vector fields (i.e.~translations of the $\R$-factors in the domain and target) we can assume that the symplectomorphism $\Phi$ above satisfies
 $$ \Phi \circ (e^t\cdot(\phi^{-t}_{\partial_s})^*)=\phi_{\partial_\tau}^t\circ \Phi, \:\: \text{on} \:\: D^*(I \times \Lambda) \cap \{s+t \in I\},$$
 i.e.~translation in the symplectisation corresponds to translation and fibre-wise rescaling in the domain $D^*(I \times \Lambda) \subset D^*L$; 
 \item If $I \colon (X,\omega) \to (X,-\omega)$ is an anti-symplectic involution that fixes $L$ set-wise, but which is fixed-point free on $L$, then $\Phi$ can be assumed to satisfy the property that $\Phi^{-1} \circ I\circ \Phi$ coincides with the anti-symplectic involution $I_0 \circ (I|_L)^* \colon D^*L \to D^*L$.
 \item When the conditions in (1) and (2) hold simultaneously, and moreover $I(\tau,y)=(\tau,I_Y(y))$ is satisfied (i.e.~the anti-symplectic involution is induced by a contactomorphism of $(Y,\ker \alpha)$ that satisfies $I_Y^*\alpha=-\alpha$), then we may assume that the conclusion of (2) holds everywhere, simultaneously as the conclusion of (1) in the subset where the the cobordism is conical. 
 \end{enumerate}
\end{thm}
\begin{proof}
The case without the additional assumptions is due to Weinstein \cite{Weinstein:Lagrangian}.

(1): Since a conical Lagrangian cobordism is a cone over a Legendrian, this can be seen to follow from the strict standard Legendrian neighbourhood theorem \cite[Theorem 6.2.2]{Geiges:Intro} by which any Legendrian $\Lambda \subset (Y,\alpha)$ has a neighbourhood which admits a contact-form preserving identification with a neighbourhood of the zero-section in $(J^1\Lambda=T^*\Lambda \times \R_z,dz+\theta_\Lambda)$. In the conical subset, one can use the fact that the cone over such a standard neighbourhood in the symplectisation is symplectomorphic to a neighbourhood of the zero-section $\R \times \Lambda \subset T^*(\R \times \Lambda)$.

(2): Since $I$ is assumed to be without fixed-points near $L$, the second part can be readily proven by the same method as the standard version of Weinstein's Lagrangian neighbourhood theorem, but after making all choices $I$-invariant. The details are left to the reader.

(3): This follows from a combination of the techniques in the proofs of (1) and (2).
\end{proof}

Here we are mainly interested in cobordisms that are symplectic surfaces, in which case the ends are cones over transverse knots. However, to produce these symplectic cobordisms, we need to go pass via constructions of Lagrangian cobordisms, in which case the ends are cones over Legendrian submanifolds.

An isotopy $T_t \subset (Y,\xi)$ of transverse knots yields a symplectic version of the trace cobordism in the following way. We may assume that the isotopy $T_t$ is constantly equal to $T_\pm$ whenever $\pm t \ge 1$. 
\begin{lma}
\label{lma:tracecob}
The \textbf{symplectic trace cobordism} from $T_-$ to $T_+$ given by
$$ \mathcal{T}_{\{T_t\}} = \{(\tau,y);\: y \in T_{v\tau} \} \subset (\R \times Y,d(e^\tau\alpha))$$
is symplectic if the speed $v >0$ is taken to be sufficiently small. Furthermore,
\begin{itemize}
 \item When $Y=U^*M$ and $T_t$ are $I_{ Y}$-invariant, the trace cobordism is $I_0$-invariant as well; and
 \item The trace cobordism depends smoothly on the isotopy $T_t$.
\end{itemize}
\end{lma}

\subsection{Symplectic push-offs of Lagrangian cobordisms}

A two-dimensional Lagrangian cobordism can be pushed-off by a $C^\infty$-small deformation to a symplectic cobordism between transverse knots. This was shown e.g.~in \cite[Lemma 4.1]{Cao}, and we recall the proof here. For completeness we give a completely general version of the construction that can be applied in any dimension.

\begin{prp}
\label{prp:pushoff}
Let $\phi \colon L \hookrightarrow (\R_\tau \times Y,d(e^\tau\alpha))$ be an even-dimensional Lagrangian cobordism from the Legendrian $\Lambda_-$ to $\Lambda_+$. Assume that $\Lambda_\pm$ admit contact forms $\alpha_\pm \in \Omega^1(\Lambda_\pm)$ for which the induced one-form on $L$ that coincides with $e^{ \tau \circ \phi}\cdot \alpha_\pm$ on the positive and negative cylindrical ends, respectively, extends to a globally defined primitive $\eta \in \Omega^1(L)$ of a symplectic form $d\eta$. Then, there exists a smooth family of cobordisms
$$\phi_t \colon L \hookrightarrow \R \times Y, \:\: \phi_0=\phi, \:\: \phi_t^*e^\tau \alpha=t\eta,\:\: t \in [0,\epsilon),$$
with conical ends, for some small $\epsilon>0$; hence $\iota_t(L)$ is symplectic for all $t \in (0,\epsilon)$. 

Moreover, if $I(\tau,y)=(\tau,I_Y(y))$ is an $\R_\tau$- equivariant anti-symplectic involution of $(\R \times Y,d(e^\tau\alpha))$ that fixes $L$ set-wise, has no fixed points on $L$, and satisfies $(I|_L)^*\eta=-\eta$, then $\iota_t$ can all be assumed to be $I$-invariant.
\end{prp}
\begin{rmk}
The one-form $\eta \in \Omega^1(L)$ makes $(L,\eta)$ into a complete Liouville cobordism. In particular, $L$ must be even-dimensional and have a non-empty convex conical end $\Lambda_+ \neq \emptyset$ holds by Stokes' theorem. 
\end{rmk}
\begin{proof}
We choose a standard Weinstein neighbourhood
$$ \Phi \colon D^*L \hookrightarrow (\R \times Y,d(e^\tau\alpha))$$
extending the Lagrangian embedding $\phi \colon L \hookrightarrow \R \times Y$
as in Theorem \ref{thm:weinstein} that is compatible with the symplectic involution. The cobordism is then simply given by the image of the section of $t\eta$ in the cotangent bundle $D^*L$.

Moreover, when $L$ is invariant under the anti-symplectic involution $I$, the sections $t\eta$ are invariant under the corresponding anti-symplectic involution $I_0 \circ (I|_L)^*$ on the cotangent bundle.
\end{proof}

\section{Conical submanifolds of cotangent bundles}

For constructing and deforming $I$-invariant submanifolds, it will be useful to first introduce the notion of a conical submanifold of a cotangent bundle. 

\begin{dfn}Consider an immersion $\gamma \colon B \looparrowright M$ and a sub-bundle $E \to B$ of the pull-back bundle $\gamma^*(T^*M) \to B$. The natural immersion $E \looparrowright T^*M$ that lifts $B \looparrowright M$ is said to be a \textbf{conical} immersion, and we denote its image by $\mathcal{C}_E \subset T^*M$. If an immersion coincides with a conical immersion near $0_M$, it is said to be \textbf{conical near $0_M$}. 
\end{dfn}
Note that conical submanifolds and immersions are invariant under $I_0$.

\begin{ex}
 \label{ex:cotangent}
 For any submanifold $B \subset M = 0_M$ there are two natural conical extensions to $T^*M$ with very different properties:
 \begin{enumerate}
 \item For the conormal bundle of $B \subset M$, i.e.~the bundle of cotangent vectors that annihilate $TB$, the corresponding conical submanifold in $T^*M$ is Lagrangian. This conical submanifold is called the \textbf{conormal bundle of $B$.}
 \item There is a canonical bundle-isomorphism of $T^*M \to M$ and the vertical sub-bundle
 $$V \coloneqq \ker\left(T(T^*M)|_{0_M} \xrightarrow{T\pi} TM\right) \cong T^*M$$
 as vector bundles over $M=0_M \subset T^*M$. For any compatible almost complex structure $J$ that interchanges the two sub-bundles $V,T0_M\subset T(T^*M)$ we get a symplectic conical submanifold given by the sub-bundle $J(TB) \subset T^*M$ along $B \subset 0_M$, i.e.~an embedding of the cotangent bundle $T^*B \subset T^*M$ of $B$. This conical submanifold will be called the \textbf{cotangent extension of $B$}.
 \end{enumerate}
 The construction also works for immersions $B\hookrightarrow M$, in which case it produces conical Lagrangian or symplectic immersions.
\end{ex}

In order to better understand the conical submanifolds near $0_M \subset T^*M$, it is often useful to pass to the symplectisation. Clearly a conical immersion is $\R_\tau$-invariant inside the coordinates $\R_\tau \times Y \cong T^*M \setminus 0_M,$ from which we conclude that
\begin{lma}
\label{lma:conicaltransverse}
Conical symplectic and Lagrangian immersions correspond to cones over submanifolds in the unit cotangent bundle that are transverse and Legendrian, respectively.
\end{lma}

\subsection{Deformation of conical symplectic surfaces}

For simplicity we here restrict our attention to two-dimensional conical symplectic submanifolds, which is the only case that we will need here. Naturally, many of these results have generalisations to conical submanifolds of higher dimension.

\begin{lma}
\label{lma:cylindrisation}
Let $\Sigma \subset T^*M$ be an $I_0$-invariant symplectic embedding of a surface of real dimension two that intersects $0_M$ cleanly in a one-dimensional link $L=\Sigma \cap 0_M$. Then $\Sigma$ is Hamiltonian isotopic through surfaces that all are
\begin{itemize}
 \item $I_0$-invariant, and
 \item intersect $0_M$ in $L$,
\end{itemize}
to a surface that is conical near $0_M$. The Hamiltonian isotopy can, moreover, be assumed to be supported in an arbitrarily small compact subset of $0_M$.
\end{lma}
\begin{proof}
Consider a conical embedding $\Sigma_0 \subset T^*M$ of $L \times \R$ that is tangent to $\Sigma$ along $L$. Hence, $\Sigma_0$ is $I_0$-invariant and symplectic. Since $C^1$-small deformations of symplectic surfaces remain symplectic, it is simply a matter of perturbing $\Sigma$ through $I_0$-invariant surfaces in a sufficiently small neighbourhood of $L$ in order to make it coincide with $\Sigma_0$ there. We can then invoke Theorem \ref{thm:HamiltonianThom} to produce the sought Hamiltonian isotopy.
\end{proof}

\begin{lma}
Let $\Sigma_i \subset T^*M$, $i=0,1,$ be two symplectic embeddings of a surface of real dimension two, where both surfaces are conical near $0_M$ and intersects $0_M$ cleanly in the same one-dimensional link $L=\Sigma_i \cap 0_M$. Then $\Sigma_0$ can be made to agree with $\Sigma_1$ near $0_M$ by a Hamiltonian isotopy $\Sigma_t=\phi^t(\Sigma_0)$ of symplectic surfaces that has support in an arbitrarily small neighbourhood of $0_M$. Furthermore, we can assume that $\Sigma_t$ all symplectic surfaces that intersect $0_M$ in the same link $L$, and which all are conical near $0_M$.
\end{lma}
\begin{proof}
Let $E_i \to L$, $i=0,1$, be the sub-bundle of $T^*M|_L \to L$ that corresponds to the conical symplectic submanifold that agrees with $\Sigma_i$ near $0_M$. The symplectic condition implies that $E_i$ are complements of the conormal bundle. Since any two such complementary bundles can be deformed through sub-bundles that are complementary to the conormal bundle, we get a smooth family of conical symplectic surfaces $\mathcal{C}_{E_t,L}$ that connects $\mathcal{C}_{E_0,L}$ and $\mathcal{C}_{E_1,L}$.

The isotopy $\mathcal{C}_{E_t,L}$ of conical symplectic surfaces induces an isotopy of $I_{ Y}$-invariant transverse knots $T_t \subset U^*M$; see Lemma \ref{lma:conicaltransverse}. Using Lemma \ref{lma:tracecob}, we produce a one-parameter family of conical $I_0$-invariant symplectic trace-cobordisms $\mathcal{T}_t \subset \R \times U^*M$ from $T_0$ to $T_t$, with a smooth dependence on $t \in [0,1]$. Furthermore $\mathcal{T}_0=\R \times T_0$.

For the induced family of concatenations of $\mathcal{T}_t$ and $\mathcal{C}_{E_t,B}$ (see Definition \ref{dfn:concatenation}), we obtain a smooth isotopy of symplectic surfaces that coincides with $\mathcal{C}_{E_t,B}$ near $0_M$. The isotopy can moreover be assumed to be supported in an arbitrarily small neighbourhood of $0_M$. Finally, the sought Hamiltonian isotopy is produced by invoking Theorem \ref{thm:HamiltonianThom}.
\end{proof}

\subsection{Construction of two-dimensional symplectic cobordisms}
\label{sec:sympconstr}

The surfaces provided by Theorem \ref{thm:main} will be constructed by concatenations of standard pieces; see Definition \ref{dfn:concatenation}. In this subsection we construct these needed standard pieces.

More precisely, we are interested in constructing two-dimensional symplectic cobordisms. This will be done by relying on constructions on exact two-dimensional Lagrangian cobordisms, which can be turned into symplectic surfaces through the push-off construction in Proposition \ref{prp:pushoff}.

However, two-dimensional Lagrangian cobordisms live inside an ambient four-dimensional symplectic manifold. Since we want to construct two-dimensional symplectic cobordisms inside six-dimensional symplectic manifolds, the surfaces that we construct inside a four-dimensional symplectic manifold will be embedded into six-dimensions via the stabilisation
$$ (\R_\tau \times Y^3 \times \{0\},d(e^\tau\alpha)) \hookrightarrow (\R_\tau \times Y^3 \times B^{2n},d(e^\tau(\alpha+x\,dy))$$
which is the symplectisation of the contact product $(Y,\alpha) \times (B^{2n},x\,dy)$, where $x\,dy=\sum^n_{i= 1} x_i\,dy_i$ is a primitive of the standard linear symplectic form.

When we also are considering an anti-symplectic involution, we will assume that it is induced by a contact-involution
$$I_Y \times I_{\C^n} \colon (Y \times B^{2n},\alpha+y\,dx) \to (Y \times B^{2n},-\alpha-y\,dx),$$
which acts while preserving the two factors, and which is given by complex conjugation in the second factor.

Before we start with the constructions we will recall the following useful identification. Consider the standard contact form on the jet-space $J^1S^{k-1}=T^*S^{k-1} \times \R_z$ which is given by $\alpha_0=dz-\theta_{S^{k-1}}$, where $\theta_{S^{k-1}} \in \Omega^1(T^*S^{k-1})$ is the tautological one-form. There is a fixed-point free contact involution 
\begin{gather*}
I_{J^1S^{k-1}} \colon (J^1S^{k-1},dz-\theta_{S^{k-1}}) \to (J^1S^{k-1},-(dz-\theta_{S^{k-1}})),\\
 T^*S^{k-1} \times \R \ni (y,z) \mapsto \left(I_0^{T^*S^{k-1}} \circ A^*(y),-z\right),
\end{gather*}
where $A \colon S^{k-1} \to S^{k-1}$ is the antipodal map, and
$$I_0^{T^*S^{k-1}} \circ A^* \colon (T^*S^{k-1},\theta_{S^{k-1}}) \to (T^*S^{k-1},-\theta_{S^{k-1}})$$
is the anti-symplectic involution obtained by composing the symplectomorphism induced by the antipodal map with fibre-wise multiplication by $-1$.
\begin{lma}[Arnold's hodograph transformation \cite{ArnoldHodograph}]
\label{lma:hodograph}
Consider the unit cotangent bundle with the contact form induced by the flat metric $(U^*\R^n,\alpha_0)$. For any linear subspace $V^k \subset T_0\R^n=\R^n$, the isotropic $k-1$-sphere $V \cap U^*_0\R^n$ has a neighbourhood that is strict contactomorphic (i.e.~the contactomorphism preserves the contact forms and not merely the contact distributions) to a neighbourhood of $j^10 \times \{0\}$ inside
$$ (J^1S^{k-1} \times T^*\R^{n-k},dz-\theta_{S^{k-1}}+\theta_{\R^{n-k}}) $$
under which
\begin{itemize}
 \item the neighbourhood intersected with $U^*_0\R^n$ is mapped into the zero-section $j^10 \times 0_{\R^{n-k}}$; and
 \item the contact-involution on $U^*\R^n$ corresponds to $I_{J^1{S^k-1}} \times I_{T^*\R^{n-k}}$, where $I_{T^*\R^{n-k}}=I_0 \circ (-\id_{\R^{n-k}})^*$ (i.e.~multiplication by $-1$ in the base $\R^{n-k}$).
\end{itemize}
\end{lma}
\begin{proof}
 Arnold's hodograph transformation takes care of the case $n=k$. The general case can be seen as follows. 
 
 The sphere $V^k \cap S^{n-1} \subset S^{n-1}$ has a neighbourhood $U \subset S^{n-1}$ that is diffeomorphic to the trivial bundle $S^{k-1} \times \R^{n-k} \to S^{k-1}$. Furthermore, under this identification, $V^k \cap S^{n-1}$ is identified with the zero-section, while the antipodal map of $S^{n-1}$ becomes the antipodal map on the base $S^{k-1}$ together with multiplication by $-1$ in the fibres $\R^{n-k}$ (i.e.~in the normal bundle of the sphere). 
 
 Finally, note that there is a canonical identification 
 $$(J^1U,dz-\theta_U) \cong (J^1 S^{k-1} \times T^*\R^{2(n-k)},dz-\theta_{S^{k-1}}+\theta_{\R^{n-k}})$$
 preserving the contact forms, and under which the involution takes the sought expression.
 \end{proof}

Next, we give the recipes for constructing the needed two-dimensional $I_0$-invariant symplectic cobordisms inside the symplectisation $\R \times U^*M \cong T^*M \setminus 0_M$ whenever $M$ is of dimension at least $\dim M \ge 2$. 

\subsubsection{Fillings of transverse knots by genus-$g$ surfaces}
\label{sec:genusgfilling}

Consider the standard Legendrian $(p,q)$ torus link $\Lambda_{(p,q)} \subset (\R^3,dz-y\,dx)$ inside the standard contact vector space. There is a Lagrangian cobordism inside $\R \times \R^3$ from $\emptyset$ to $\Lambda_{p,q}$ which is of genus $\frac{1}{2}(p-1)(q-1)$; see e.g.~\cite{Ekhoka}. Since every contact manifold is locally contactomorphic to the contact vector space, we can now apply Proposition \ref{prp:pushoff} combined with the method of Subsection \ref{sec:sympconstr} to construct symplectic cobordisms of any genus from $\emptyset$ to the unique null-homotopic transverse knot in the contactisation of any contact manifold of dimension at least five. (Recall Proposition \ref{prp:hprinciple} by which there is a unique null-homotopic transverse knot in connected high-dimensional contact manifolds.)

Since the aforementioned cobordism can be constructed inside the symplectisation of any arbitrarily small neighbourhood of a contact manifold, it is easy to construct an $I_0$-invariant such cobordism consisting of two disjoint genus-$g$ surfaces. 

\subsubsection{Connected $I_0$-invariant filling without fixed points}
\label{sec:fillingII}
Here we construct a connected $I_0$-invariant filling in the symplectisation of $J^1S^1$ on which $I_0$ has no fixed points. First, consider the two-copy link
$$j^1(1) \sqcup j^1(-1) \subset (J^1S^1,dz-\theta_{S^1})=(S^1 \times \R_p \times \R_z,dz-p\,d\theta),$$
which is invariant under the contact-involution $I_{J^1S^1}(\theta,p,z)=(\theta+\pi,-p,-z)$. This Legendrian has a connected Lagrangian $I_{J^1S^1}$-invariant filling in $\R_\tau \times J^1S^n$ given by
$$\left\{ (\gamma_1(t),j^1\gamma_2(t)) \in \R_\tau \times J^1S^n \right\}$$
for some smooth curve $(\gamma_1,\gamma_2)(t) \in \R^2$ which satisfies
\begin{itemize}
 \item $\gamma_1(t)=\gamma_1(-t)$;
 \item $\gamma_1''(t) \ge 0$; and 
 \item $\gamma_1(t)=|t|$ for $t \gg 0$,
\end{itemize}
while
\begin{itemize} 
 \item $\gamma_2(t)=-\gamma_2( - t)$;
 \item $\gamma_2'(t) \ge 0$; and
 \item $\gamma_2(\pm t)=\pm 1$ for $t \gg 0$.
\end{itemize}

 Note that $j^1 \gamma_2(t)=\{z=\gamma_2(t)\} \subset J^1S^1$ is a push-off of the zero-section $j^10$ by the time-$\gamma_2(t)$ Reeb flow. Hence, the pull-back of $e^\tau(dz-p\,d\theta)$ under the map $(\gamma_1,j^1\gamma_2)$ is clearly an exact one-form, which means that this map is an exact Lagrangian embedding.

Again, since $J^1S^1$ embeds into the standard contact Darboux ball, we can argue as in Subsection \ref{sec:genusgfilling} to produce a symplectic $I_0$-invariant cylinder in the symplectisation $\R \times U^*M$ whenever $\dim M \ge 2$.

\subsubsection{ Pairs of boundary connected sums via invariant pairs of one-handles}
\label{sec:connectedsumI}

 Here we describe how one can perform a Lagrangian boundary connected sum operation inside the symplectisation via a pair of Lagrangian one-handles that are exchanged under the involution $I_0$. Thus, both the one-handles and the connected sum operations are $I_0$-invariant. The construction is as follows: 

One can perform the so-called Legendrian ambient connected-sum operation on two disjoint Legendrians $\Lambda_0 \sqcup \Lambda_1 \subset (Y,\alpha)$ as described in work \cite{Dimitroglou:Ambient} by the second author. Furthermore, there is a Lagrangian standard one-handle attachment cobordism from the union $\Lambda_0 \sqcup \Lambda_1$ to the result $\Lambda_3$ of the connected sum. Concatenating the cobordism $\R\times (\Lambda_0 \sqcup \Lambda_1)$ with this Lagrangian one-handle amounts to performing a boundary connected sum. 

In the case when $\Lambda_0 \sqcup \Lambda_1 \cap I_Y(\Lambda_0 \sqcup \Lambda_1 ) =\emptyset,$ we can use a sufficiently thin Lagrangian handle-attachment cobordism to ensure that also the cobordism is disjoint from its own image under the anti-symplectic involution $I_0$. This is the sought $I_0$-invariant pair of Lagrangian one-handles that give rise the two simultaneously performed boundary connected sums.

Finally, arguing as above, one can readily turn this construction into a boundary-connected sum operation on symplectic surfaces in symplectisations.

\section{Conical surfaces in the complex projective space}

For constructing $I$-invariant symplectic submanifolds of $\CP^n$ it is useful to symplectically identify a neighbourhood of the fixed-point locus $\RP^n \subset \CP^n$ of $I$ with a neighbourhood of $0_{\RP^n} \subset T^*\RP^n$. There are many such identification, e.g.~by Weinstein's neighbourhood theorem. We will follow a path similar to the one by Adaloglou \cite{Adaloglou}, in order to obtain an identification of a neighbourhood as large as possible.

The standard projective plane $\RP^n \subset \CP^n$ is contained inside the complement
$$ \RP^n \subset \CP^n \setminus Q^{n-1} $$
of the standard smooth quadric
$$ Q^{n-1}=\{[Z_0:Z_1:\ldots:Z_n];\:Z_0^2+Z_1^2+\ldots+Z_n^2=0\}.$$
Note that the involution $I$ fixes $Q^{n-1}$ as a subset, and thus descends to a fixed-point free anti-symplectic involution on $Q^{n-1}$.

Note that $\RP^n$ has isometry group $PO(n+1)$ that admits a natural embedding $PO(n+1) \subset PU(n+1)$ as a subgroup that commutes with $I$. This subgroup acts linearly by $O(n+1)$ on the homogeneous coordinates $\C^{n+1} \setminus \{0\}$. The group $PO(n+1)$ of real isometries also fixes the quadric $Q^{n-1}$ set-wise. 

\subsection{The symplectic structure on the complement of a smooth quadric}

A folklore result states that the complement of a smooth quadric $Q^n \subset \CP^n$ is symplectomorphic to a fibre-wise convex subset of the cotangent bundle $T^*\RP^n$; see work by Adaloglou for a detailed proof \cite[Corollary 2.4]{Adaloglou}. This claim is very closely related to the even more well-known folklore result that the affine quadric is symplectomorphic to $T^*S^n$; indeed, taking the double cover of $\CP^n$ branched along the quadric $Q^{n-1}$ we get a projective quadric, with the complement of the branching locus equal to a subset of the double cover $T^*S^n \to T^*\RP^n$. Since we will need the identification to be compatible with the anti-symplectic involutions and $PO(n+1)$-actions, we reprove the existence of the symplectomorphism here. 

Through each point $p \in \RP^n$ we can consider the element in $PU(n+1)=Isom(\CP^n,\omega_{FS})$ which is uniquely determined by the property that it fixes the point $p$, while its differential is given by multiplication by $i$ in the tangent space $T_p\CP^n$. The image of $\RP^n$ under this linear map is thus a uniquely determined Lagrangian embedding $\RP^n_p \subset \CP^n$ of the projective plane that intersects the standard $\RP^n$ transversely in the point $p$.

The Lagrangian $\RP^n_0$ which passes through the origin $0 \in \RP^n$ in the standard affine chart is given by the imaginary projective plane $\RP^n_0=i\RP^n$. The latter can be seen to intersect $\RP^n$ in the projective plane $\RP^{n-1}_\infty \subset \RP^n$ at infinity with respect to the standard affine chart. Note that
$$PO(n+1)\subset PU(n+1) =Isom(\CP^n,\omega_{FS})$$
acts transitively on $\RP^n$, as well as on the set of Lagrangian subspaces $\RP^n_p$ for $p \in \RP^n$.

Note that $\RP^n_0 \cap Q^{n-1}$ is equal to the Lagrangian $(n-1)$-dimensional sphere
$$\{z_1^2+\ldots+z_n^2=
-1,\:\mathfrak{Re}\mathbf{z}=0\}\subset \{1+z_1^2+\ldots+z_n^2=0\}$$
inside the symplectic submanifold $Q^{n-1} \subset \CP^n$. Furthermore, any real line through $0 \in \RP^n_0$ is determined by its intersection in two antipodal points on the latter sphere.

\begin{prp}
The open Lagrangian $n$-discs $B^n_p \subset \RP^n_p$ whose boundary is given by the $(n-1)$-dimensional sphere
$$\partial \overline{B^n_p} = Q^{n-1} \cap \RP^n_p \subset \RP^n_p$$
are embedded and foliate $\CP^n \setminus Q^{n-1}$. In other words, $\CP^n \setminus Q^{n-1}$ is a tubular neighbourhood of $\RP^n$ in the form of a Lagrangian $n$-disc bundle on which $PO(n+1)$ acts.
\end{prp}
\begin{proof}
\emph{The fact that the discs form an embedded tubular neighbourhood:} Two discs $B^n_p$, $B^n_q$ for $p \neq q$ are disjoint for the following reason. Both discs are foliated by real lines in $\RP^n_p$ and $\RP^n_q$ that pass through $p$ and $q$, respectively, whose complexifications moreover have intersections with $\RP^n$ that again are real lines. The two real lines $\ell_p \subset \RP^n_p$ and $\ell_q \subset \RP^n_q$ passing through $p$ and $q$, respectively, can be seen to have an intersection in $B^n_p \cap B^n_q$ if and only if their complexified lines coincide. Indeed, if the corresponding real lines $\C\ell_p \cap \RP^n \subset \RP^n$ and $\C\ell_q \cap \RP^n \subset \RP^n$ intersect at an imaginary point, then these lines must coincide.

Thus, if $B^n_p$ and $B^n_q$ intersect for $p \neq q$, we are in a situation where the complexifications of $\ell_p$ and $\ell_q$ coincide, and intersect $\RP^n$ in the unique real line that passes through both $p$ and $q$. When this holds one can verify, e.g.~by consulting the case $n=1$, that $\ell_p$ and $\ell_q$ for two different points $q \neq p$ that live on a common line in $\RP^n$ intersect precisely at the two points in the generic intersection $\ell_p \cap Q^{n-1}=\ell_q \cap Q^{n-1}$. This contradiction implies that the two balls $B^n_p$ and $B^n_q$ must be disjoint, as sought. 

\emph{The fact that the discs make up all of $\CP^n \setminus Q^{n-1}$:} This follows from elementary topology. It suffices to note that $\CP^n \setminus Q^{n-1}$ is open and connected. The disc bundle constructed is also open, and its topological boundary is contained inside $Q^{n-1}$. 
\end{proof}

We will now compare the symplectic structure on $\CP^n \setminus Q^{n-1}$ with the symplectic structure on the cotangent bundle $T^*\RP^n$. The Fubini--Study form can equivalently be written as
$$ \omega_{FS}=\frac{i}{2}\del\delbar\rho, \:\: \rho(\mathbf{z})=\log{\left(\frac{1+\|\mathbf{z}\|^2}{|1+z_1^2+\ldots +z_n^2|}\right)},$$
in the complement of $Q^{n-1}$ in the standard affine coordinates. The function $\rho(\mathbf{z})$ extends to a well-defined potential on all of $\CP^n \setminus Q^{n-1}$. (It is induced by a trivialisation of the line-bundle $\mathcal{O}(2) \to \CP^n$ defined in the complement of $Q^{n-1}$.) The fact that $\rho$ can be extended to a potential away from $Q^{n-1}$ can also be checked explicitly, e.g.~by observing that it is has the same expression in all affine charts.

Furthermore, this potential is preserved by the action of $PO(n+1)$. This can be seen by considering the homogeneous function
$$ \log {\left(\frac{\|\mathbf{Z}\|^2}{|Z_0^2+Z_1^2+\ldots +Z_n^2|}\right)}$$
of degree zero which is invariant under $O(n+1)$ and well-defined away from $\{Z^2_0+Z^2_1+\ldots+Z^2_n=0\}$.
In this manner, we get a primitive $-\frac{1}{4}d^c\rho \in \Omega^1(\CP^n\setminus Q^{n-1})$ of the Fubini--Study form defined on $\CP^n \setminus Q^{n-1}$.
\begin{thm}
\label{thm:identification}
There exists a $PO(n+1)$-equivariant embedding
$$\Phi \colon \left(\CP^n \setminus Q^{n-1},-\frac{1}{4}d^c\rho\right)\hookrightarrow \left(T^*\RP^n,\theta_{\RP^n}\right)$$
that preserves the primitives of the symplectic forms, which is uniquely determined by the properties:
\begin{enumerate}
\item the restriction to $\RP^n \subset \CP^n$ is the canonical identification with the zero-section $0_{\RP^n} \subset T^*\RP^n$;
\item each $B^n_p$ is mapped into the corresponding cotangent fibre $T^*_p\RP^n$;
\item the anti-symplectic involution $I$ in the domain is identified with $I_0$ in the cotangent bundle, i.e.~the fibre-wise scalar multiplication by $-1$ in the target; and
\item when restricted to any $B^n_p$, the map preserves lines through the origin (as defined by the conformal structure in the domain and the linear bundle structure in the target);
\item when restricted to any $B^n_p$, the map sends concentric spheres in $B^n_p$ around the origin (defined by the Fubini-Study metric on $\CP^n$), to a spherical cotangent fibres of constant radius in $T^*_p\RP^n$ (induced by the Fubini--Study metric on $\RP^n$).
\end{enumerate}
\end{thm}
\begin{proof}
For the uniqueness, we note that any symplectomorphism $\Phi$ of $T^*M$ that is the identity along the zero-section and preserves fibres must be the identity. Indeed, conjugating the symplectomorphism with a fibre-wise rescaling
$$\sigma_t^{-1} \circ \Phi \circ \sigma_t, \:\sigma_t(x)=t\cdot x,$$
yields a Hamiltonian isotopy of $T^*M$ that fixes both the zero-section pointwise, and the Lagrangian fibres set-wise. Since such a Hamiltonian must be constant both along the zero-section as well as along the fibres, it follows that the generating Hamiltonian is globally constant. Consequently, $\Phi$ commutes with $\sigma_t$ for all $t >0$. Furthermore, since $D\Phi$ must be the identity along $0_M$, it follows that $\Phi=\id_{T^*M}$ holds everywhere, as sought.

In order to define the $P(O+1)$-equivariant map it suffices to prescribe its restriction to a map
$$ \Phi_0 \colon B^{n}_0 \hookrightarrow T^*_0\RP^n.$$
First note that the $O(n) \subset P(O+1)$ equivariance implies that this map is determined once we know it on the ray 
$$\{y_i=0;\: i=2,\ldots,n\} \cap B^{n}_0 \subset \Im\C^n \subset \C^n =\CP^n \setminus \CP^{n-1}_\infty.$$
We prescribe this ray to be sent into the one-dimensional subspace consisting of this cotangent vectors that annihilate the tangent space to the embedding $\{x_1=0\} \subset \RP^{n-1} \subset \RP^n$ at $0 \in \RP^n$.

Defining the map $\Phi_0$ with some more care along the line, we can ensure that the primitive $-\frac{1}{4}d^c \rho$ of $\omega_{FS}$ is the pull-back of $\theta_{\RP^n}$ under the restricted map
$$ \CP^1 \setminus Q^0 \hookrightarrow T^*\RP^1 \subset (\T^*\RP^n,\theta_{\RP^n}),$$
where $\CP^1 \setminus Q^0$ is the complexified real line $\C\{y_i=0,\: i=2,\ldots,n\} \setminus Q^{n-1}$ with the quadric removed. Here the latter complexified line is endowed with the restriction of the Fubini--Study form (which again simply is the Fubini--Study form in dimension one).

In fact, since the primitive of the symplectic form is given by
$$ -\frac{1}{4}d^c \rho |_{\RP^n_0}=\frac{1}{2}\sum_i\left( \frac{-1}{1+\|\mathbf{z}\|^2} +\frac{-1}{|1+z_1^2+\ldots+z_n^2|}\right)y_i\,dx_i=\sum_i \frac{-y_i}{ |1-\|\mathbf{y}\|^4|} dx_i,$$
 which when further restricted to $B^n_0=\{\|\mathbf{y}\|^2<1\}\subset \RP^n_0$ becomes
$$-\frac{1}{4}d^c \rho |_{B^n_0} = \sum_i \frac{-y_i}{1-\|\mathbf{y}\|^4} dx_i.$$
It now follows by construction that $-\frac{1}{4}d^c \rho$ is the pull-back of $\theta_{\RP^n}$ on all of $T\CP^n$ along $\CP^1 \setminus Q^0 $ (not just on the tangent space of the latter). It thus follows that $\Phi$ is symplectic everywhere.

(1), (2), (3), (4): These points now follow immediately by construction.

(5): follows by (4) and the $PO(n+1)$-equivariance. Note that the action of $O(n) \subset PO(n+1)$ acts on $B^n_0$ while preserving the concentric spheres around the origin.
\end{proof}

\subsection{Definition of conical submanifolds}

By using the above symplectic identification of $\CP^n \setminus Q^{n-1}$ with a fibre-wise convex subset of $T^*\RP^n$ we can make the following definition.

\begin{dfn}
An immersion in $\CP^n$ is said to be \emph{conical near $\RP^n$} if it corresponds to an immersion in $T^*\RP^n$ which is conical near $0_{\RP^n}$ under the symplectic embedding $ \CP^n \setminus Q^{n-1}\hookrightarrow T^*\RP^n$ from Theorem \ref{thm:identification}.
\end{dfn}

Note that immersions that are conical near $\RP^n$ in particular are $I$-invariant in some neighbourhood of $\RP^n$, and determined by an $I$-invariant transverse link in $U^*\RP^3$ over which the surface is a cone.

\section{Proof of Theorem \ref{thm:main}}
We start by recalling the constructions of the standard $I$-invariant algebraic curves $\Sigma_i$ of degrees $i=1$ and $i=2$. We will use parts of these for constructing our general $I$-invariant symplectic surfaces.

The standard $\RP^1 \subset \RP^3$ as well as the standard null-homologous unknot
$$\{x_1^2+x_2^2=1, x_3=0\} \subset \RP^3$$
are both real-algebraic. The complex locus $\CP^1$ of $\RP^1$ is conical on the nose when considered inside $\CP^3\setminus Q^2 \hookrightarrow T^*\RP^3$. Denote this surface by $\Sigma_1=\CP^1$ and note that it is of degree one.

It is no longer true that the algebraic surface $\Sigma_2$ with real locus the unknot $\{x_1^2+x_2^2=1, x_3=0\}$ is conical in a neighbourhood of $\RP^3$; however, we can use Lemma \ref{lma:cylindrisation} to Hamiltonian isotope the complex part of this algebraic knot while fixing the real locus, in order to make it conical near $\RP^3$. Denote the obtained conical $I$-invariant symplectic surface by $\Sigma_2$, which is of degree two.

Since $\Sigma_i$ are conical near $\RP^3$, their intersections with some small tubular neighbourhood $U \subset \CP^3$ of $\RP^3 \subset \CP^3$ are surfaces that are conical in $T^*\RP^3$ under the identification in Theorem \ref{thm:identification}. Moreover, we can assume that the latter identification sends $U$ to a radius-$r$ co-disc bundle $D^*_rT^*\RP^3$ for some small radius $r>0$.

In particular, $(\Sigma_i \cap U) \setminus \RP^3$ corresponds under Theorem \ref{thm:identification} to a cone $(-\infty,a) \times T_i$ over a transverse $I$-invariant link $T_i \subset U^*\RP^3$ consisting of two components that are interchanged by $I$. In the case $i=1$ each of the two components lives in the non-trivial homotopy class in $\pi_1(U^*\RP^3)\cong\pi_1(\RP^3)\cong \Z_2,$ while in the case $i=2$ each component lives in the trivial homotopy class.

Moreover, we will write
$$ \Sigma_i^{\OP{hat}} \coloneqq \Sigma_i \setminus U \subset \CP^3$$
for the complement of the conical sub-surface $U\cap \Sigma_i \supset (\Sigma_i)_\R$. Note that $\Sigma_i^{\OP{hat}}$ is symplectic, embedded, $I$-invariant, with boundary given by the transverse link $T_i$.

Now consider an arbitrary smooth link $L \subset \RP^3$. First apply the construction in Part (2) of Example \ref{ex:cotangent} to produce a piece of an $I$-invariant symplectic surface near $\RP^3$ that is conical near $\RP^3$. Denote this surface by $\Sigma^0_L$, which is conical near $\RP^3$ and intersects $\RP^3$ in the link $L$. Note that the boundary of $\Sigma^0_L$ can be identified with a transverse link in $U^*\RP^3$ that consists of an even number of components that are interchanged under $I$. Choosing an orientation of $L$ trivialises this action by singling out representatives $T^{0,+}_L$ for each equivalence class, so that $\partial \Sigma^0_L=T^0_L=T^{0,+}_L \sqcup I(T^{0,+}_L)$.

Then we concatenate (see Definition \ref{dfn:concatenation}) the surface $\Sigma^0_L$ with suitable $I$-invariant symplectic cobordisms coming from the connected-sum construction as described in Subsection \ref{sec:connectedsumI}. We can e.g.~make the choice of connecting all components of $T^{0,+}$, and then extending this connected symplectic cobordism to one which is $I$-invariant and consisting of two components. After suitable choices of boundary connected sums, we obtain an $I$-invariant symplectic surface $\Sigma^1_L$ with boundary consisting of two transverse links that are interchanged by $I$, has $L$ as its fixed-point locus, and which is of genus $\OP{rk} H_0(L)-1$.

The classification of $I$-invariant transverse knots from Proposition \ref{prp:hprinciple} implies that the boundary of $\Sigma^1_L$ is transverse isotopic inside $U^*\RP^3$ to either $T_1$ or $T_2$ through $I$-invariant transverse knots, depending on its homotopy class. First adjoining the symplectic $I$-invariant trace cobordism produced by Lemma \ref{lma:tracecob} and then adjoining the $I$-invariant ``hat'' produced above $\Sigma^{hat}_i$ of a suitable degree $i \in \{1,2\}$, yields a symplectic embedding of a closed connected surface of genus $\OP{rk} H_0(L)-1$ that is $I$-invariant and intersects $\RP^3$ in $L$. This finishes the construction of a connected type-I surface of genus $\OP{rk} H_0(L)-1$.

In order to construct a surface of type-I of higher genus we start by proceeding as above, constructing the $I$-invariant surface $\Sigma^1_L$ with two transverse boundary components, and which is of genus $\OP{rk} H_0(L)-1$. We then use the connected symplectic genus-$k$ filling from Subsection \ref{sec:genusgfilling} and double it with the involution $I$, after which we perform pairs of boundary connected sums as in Subsection \ref{sec:connectedsumI} to make the resulting surface connected and invariant under $I$. Note that the homology class of the boundary of the surface is not altered by these operations, but that the process increased the genus of the surface $\Sigma^1_L$ by $2k$. To finish, we add the $I$-invariant ``hat'' i.e.~ $\Sigma_i^{\OP{hat}}$ of a suitable degree $i \in \{1,2\}$ in order to close up the surface. 

Finally, for the construction of surfaces of type-II, the construction is similar to the case treated in the previous paragraph. In this case, however, instead of adding a pair of genus $k$-surfaces, we add $k$ disjoint copies of the connected $I$-invariant cylinder from Subsection \ref{sec:fillingII} (pushed-off to a symplectic surface), and perform $k$ pairs of boundary connected sums with these cylinders as described in Subsection \ref{sec:connectedsumI}. Note that this process increases the genus of $\Sigma^1_L$ by $k$. 
\qed

\section{Natural questions}

In future work we plan to study the following questions. First, it is very natural to consider the isotopy problem for the symplectic flexible links.
\begin{quest}
Are two flexible symplectic links with the same topological invariants Hamiltonian isotopic through flexible symplectic links?
\end{quest}
We expect that an $I$-equivariant version of the $h$-principle for high-codimension symplectic embeddings can be used to give a positive answer to the this question; see \cite{Eliashberg:IntroductionH} for the relevant $h$-principle. The smooth but non-symplectic case was treated by the first author in \cite{Johan:Flexible}.

Second, it would be useful to find a concept that is less flexible than symplectic $I$-invariant surfaces, but more flexible than real-algebraic curves. In the light of the constructions that were used in the present article, we make a tentative proposal of a definition.

Call an $I$-invariant piecewise-smooth embedded closed surface $\Sigma \subset \CP^3$ \textbf{$I$-invariant symplectic-conical surface with a holomorphic hat} if there exists an $I$-invariant neighbourhood $U \subset \CP^3$ of $\RP^3$ with contact-type boundary such that $\Sigma$ is holomorphic in $\CP^3\setminus U$, while it is smooth, symplectic, and conical inside $U$, and moreover has complex tangent planes along its intersection with $\RP^3$.
\begin{quest}
What are the restrictions on the link type in $\RP^3$ obtained as the real part of an $I$-invariant symplectic-conical surface with a holomorphic hat, when the surface is of degree at most two?
\end{quest}
The above question can be reformulated as a question about the existence of holomorphic hats inside $\CP^3 \setminus \Phi^{-1}(D^*_{<r}\RP^3)$, where $\Phi$ is the symplectic embedding from Theorem \ref{thm:identification}, where the holomorphic hat has boundary equal to a certain transverse link inside the boundary of $\Phi^{-1}(U^*_r \RP^3)$ of a tubular neighbourhood of $\RP^3$ symplectomorphic to a radius-$r$ unit sphere bundle. Symplectic hats inside $\CP^2$ were studied in work by Etnyre--Golla \cite{Hats}

Also note that the image of a symplectic-conical surface with a holomorphic hat under a real-algebraic knot projections to $\CP^2$ is an $I$-invariant immersions of the same type. The knot projection does not preserve the symplectic property of a submanifold in general; however, if we assume that the surface in $\CP^3$ is conical with complex tangents near $\RP^3$, its projection is an immersed symplectic-conical surface with a holomorphic hat. Indeed, the image of a (real) line under the knot projection is again a (real) line. Thus, when a surface which is conical near $\RP^3$ moreover has complex tangent planes along $\RP^3$, the cones of the aforementioned types are tangent to a complexified real lines along each of its fibres.

\bibliographystyle{unsrt}
\bibliography{references}
\end{document}